\documentclass[dvips, 12pt, a4paper]{paper}
\usepackage{amsmath,amssymb,amsthm, dsfont}
\usepackage{color}
\definecolor{violet}{rgb}{0.5,0,0.8}

\theoremstyle{plain}
\newtheorem{thm}{Theorem}[section]

\theoremstyle{plain}
\newtheorem{cor}[thm]{Corollary}

\theoremstyle{plain}

\theoremstyle{plain}

\theoremstyle{plain}

\theoremstyle{plain}
\newtheorem{proposition}[thm]{Proposition}

\theoremstyle{plain}
\newtheorem{definition}[thm]{Definition}

\theoremstyle{plain}

\theoremstyle{plain}
\newtheorem{remark}[thm]{Remark}

\begin{document}
\markboth{A. Calzolari, B. Torti}{Thin-thick approach to martingale representations on progressively enlarged filtrations}
\title{Thin-thick approach to martingale representations on progressively enlarged filtrations}

\author{Antonella Calzolari \thanks{Dipartimento di Matematica - Universit\`a di Roma
``Tor Vergata'', via della Ricerca Scientifica 1, I 00133 Roma,
Italy }  \and Barbara Torti $^*$}
\maketitle

\begin{abstract}
In this work, we investigate the propagation of the (strong) predictable representation property from a general reference filtration $\mathbb{F}$ to its progressive  enlargement   with a random time $\tau$.~We assume that the immersion property holds, that the $\mathbb{F}$-thin part of $\tau$  is finite with positive probability, and all elements of its exhausting sequence  are $\mathbb{F}$-predictable stopping times.~In this setting,
 we prove a martingale representation theorem  on the enlarged filtration valid across the entire time axis.~Then, in light of the obtained result, we review some well-known credit risk models.
\end{abstract}


\section{Introduction}	
\label{sec1}
Let $\mathbb{F}$ be a filtration and  $(M^j)_{j\in J}$ a finite set of square-integrable pairwise orthogonal $\mathbb{F}$-martingales.~Assume that any
square-integrable $\mathbb{F}$-martingale starting from zero can be represented as a sum of stochastic integrals of predictable processes
w.r.t.~the elements of $(M^j)_{j\in J}$.~In other words, $(M^j)_{j\in J}$ is an $\mathbb{F}$-\textit{basis} (see \cite{davis1}).~As a consequence, $(M^j)_{j\in J}$ has the \textit{(strong) predictable representation property} w.r.t.~$\mathbb{F}$ ($\mathbb{F}$-PRP), that is all $\mathbb{F}$-local martingales can be represented via stochastic integration  of predictable processes w.r.t.~the elements of $(M^j)_{j\in J}$ (see \cite{he-wang-yan92}).\bigskip\\
Let $\tau$ be any random time and $\mathbb{F}^\tau$ the minimal filtration under the usual conditions containing $\mathbb{F}$ and making $\tau$ a stopping time, that is the \textit{progressive enlargement of $\mathbb{F}$ with
$\tau$}.~It makes sense to consider the \textit{$\mathbb{F}^\tau$-compensated occurrence process of $\tau$}: it deals with the $\mathbb{F}^\tau$-martingale $H^{\tau,\mathbb{F}^\tau}$  obtained as Doob's compensation of the sub-martingale $\mathbb{I}_{\tau\leq \cdot}$ or equivalently the $\mathbb{F}^\tau$-martingale defined by
\begin{equation}\label{eq-comp}
H^{\tau,\mathbb{F}^\tau}:=\mathbb{I}_{\{\tau\leq\cdot\}}-B^{\tau,\mathbb{F}^\tau},
\end{equation}
where $B^{\tau,\mathbb{F}^\tau}$ is the $\mathbb{F}^\tau$-compensator of $\tau$.\bigskip\\
According to the \textit{thin-thick decomposition} of random times (see \cite{ak-chou-jean-18}), $\tau=\tau_1\wedge\tau_2$, $\tau_1\vee\tau_2=\infty$, where the graph of the \textit{thin part} $\tau_1$ is contained in the disjoint union of graphs of a sequence of $\mathbb{F}$-stopping times $(T_n)_{n\geq 1}$, and the \textit{thick part} $\tau_2$ is a random time such that $P(\tau_2=T<\infty)=0$ for every $\mathbb{F}$-stopping time $T$.~If $\tau_1=\infty$ a.s.~then $\tau=\tau_2$ a.s.~and the
graph of $\tau$ is disjoint from the graph of any $\mathbb{F}$-stopping time (\textit{$\mathbb{F}$-avoidance}).\bigskip\\
In this paper, we focus on the $\mathbb{F}^\tau$-PRP when $P(\tau_1<\infty)>0$ (\textit{$\tau_1$ nontrivial}).~More precisely, we provide a set of sufficient conditions so that the family of martingales $\left\{(M^j)_{j\in J},
H^{\tau,\mathbb{F}^\tau}\right\}$ is an $\mathbb{F}^\tau$-basis.~\bigskip\\
Our main theorem states that, when $\tau_1$ is nontrivial, if
\begin{itemize}
  \item [(i)] $\mathbb{F}$ is immersed in $\mathbb{F}^\tau$, that is $\mathbb{F}$-martingales are also $\mathbb{F}^\tau$-martingales,
      \item [(ii)]  $(T_n)_{n\geq 1}$ is a sequence of $\mathbb{F}$-predictable stopping times,
      \item [(iii)]  for any $n\geq 1$, $\mathcal{F}_{T_n}=\mathcal{F}_{T_n^-}$,
      \end{itemize}
      then $\left\{(M^j)_{j\in J},
H^{\tau,\mathbb{F}^\tau}\right\}$ is an $\mathbb{F}^\tau$-basis.\bigskip\\
\noindent The main contribution of our result is that here $\tau$ is allowed to overlap $\mathbb{F}$-stopping times, whereas, in the papers dealing with the same problem,  $\mathbb{F}$-avoidance is assumed with a few exceptions.~We point out that, under $\mathbb{F}$-avoidance, assumptions (ii) and (iii) do not make sense, but, as shown in \cite{ditella-eng-22}, assumption (i) is enough for getting the thesis.\bigskip\\
 In the literature, working on martingale representations on progressively enlarged filtrations, different assumptions have been considered not only on the nature of the random time $\tau$ (avoiding or not) or on the filtration to be expanded (for example, supporting continuous local martingales only, or generated by a Poisson process, or by a more general L\'evy process, or by a step process), but also on the class of local martingales to be represented (for example, $L^p$-martingales for some $p\ge 2$, or martingales in form of conditional expectations of processes at time $\tau$).~Often immersion condition has been weakened, assuming that all local martingales w.r.t.~the reference filtration are just semi-martingales w.r.t.~the expanded filtration.~Furthermore, several authors limited the representation up to time $\tau$ and others dealt separately with the representation from time $\tau$ onwards, fewer dealt with the representation on the entire time axis.~Moreover, some authors investigated weak martingale representations, that is representations by means of integrals w.r.t.~compensated random measures or, equivalently, by means of integrals of optional processes w.r.t.~a set of martingales.~As well-known, predictable representations are optional ones, but in general not vice versa, and if the latters can be useful in the BSDE context,
 the formers fit better in financial applications, since admissible hedging strategies are modeled as predictable processes.~We limit ourselves to recall some recent papers that do not assume the avoidance property.~In  \cite{ak-jean-rut-19} the focus is on the case where $\mathbb{F}$ is a Poisson filtration and predictable representations are provided for some classes of $\mathbb{F}^\tau$-martingale stopped at $\tau$.~In \cite{choulli-Dave-Van-20} and \cite{choulli-Alha-23} the authors study optional representations, in the first paper up to time $\tau$ only and in
  the second from $\tau$ onwards only.~In \cite{ditella-jean-21} and \cite{caltor22} $\mathbb{F}$ is the natural filtration of a point process and of a multivariate point process,   respectively, and the main results deal with the propagation of martingale representations to the progressive expansion of $\mathbb{F}$ by means of a different point process and by means of a multivariate point process,  respectively.~Finally, in \cite{jean-song15}, the authors investigate, as we do,  predictable representations on $\mathbb{F}^\tau$ without limitations on the nature of $\mathbb{F}$ and over the entire
  time axis.~Here, we use a different methodology, which has the advantage of suggesting easily interpretable hypotheses.~Indeed, we split the problem of the identification of an $\mathbb{F}^\tau$-basis into two steps: first, the construction of a basis on an intermediate filtration obtained progressively enlarging $\mathbb{F}$ with the \textit{thin part of $\tau$}; then, the construction of a basis on the progressive enlargement of the intermediate filtration with the \textit{thick part of $\tau$}.~Solving the first issue is the real goal of our work.~In fact the second step turns out to be a martingale representation problem in progressive enlargement under immersion and avoidance.~As far as we know, our result is new and cannot be covered by the existing literature.\bigskip\\
In our work we were motivated by the interest in credit risk modeling,
where $\tau$ plays the role of the \textit{default}.~More precisely, since PRP is equivalent to completeness of the market, we looked for new conditions under which the insider
can complete its extended market, and for the expression of the assets to add.~Actually, our theorem applies well to the \textit{hybrid models} considered first in \cite{MR2070167} and then  in \cite{jiao-li18},  \cite{MR3758922} and \cite{fontana-smith17}, when $\tau$ is the minimum between the occurrence time of a fully unexpected event and the hitting time of a threshold modeled as a finite random variable.~Then, using our result the insider can complete its portfolio by adding a tradable asset.~The new asset can be explicitly written in terms of  the \textit{intensity rate} of the fully unexpected event, and of the probability that $\tau$ coincides with the hitting time of any threshold's level given the history strictly prior to it.\bigskip\\
This paper is organized as follows.~In Section \ref{sec2}, we fix the notations and recall some basic definitions.~In Section 3, we introduce in great generality the fundamental hypotheses we will use on  $\tau$ and underline their relevant consequences.~In Section \ref{sec4}, we provide the expression of the $\mathbb{F}^\tau$-compensator of $\tau$ and prove the main theorem.~In Section \ref{sec5}, we discuss some examples of application to credit risk.~Finally, we devote Section \ref{sec6} to conclusions and research proposals.
\section{Notations and definitions}\label{sec2}
Let  $(\Omega, \mathcal{F}, P)$ be a complete probability space and  $\mathbb{A}$ a filtration under the usual conditions on it.~$\mathcal{M}^2(P,\mathbb{A})$ denotes the Hilbert space of square-integrable $(P,\mathbb{A})$-martingales with inner product  $(M^j,M^l)\rightarrow E^P[M^j_\infty M^l_\infty]$ ($\mathcal{M}_0^2(P,\mathbb{A})$ when the initial value is null).~For the fundamental definitions we mainly refer to  \cite{aksamit-jean-17}.\bigskip\\
Let us introduce the notion of \textit{basis} of a filtration in the sense of Davis and Varaiya (see \cite{davis1}).
 \begin{definition}\label{def:basis}
An $\mathbb{A}$-basis is a subset $(M^j)_{j\in J}, J\subset \mathbb{N},$ of $\mathcal{M}^2(P,\mathbb{A})$ whose elements are pairwise strongly orthogonal martingales
 such that each $V\in \mathcal{M}^2(P,\mathbb{A})$ satisfies
\begin{align*}
V_t=V_0+\sum_{j\in J}\int_0^t \phi^V_j(s) d M^j(s),\;\;t\geq 0,
\end{align*}
where $V_0$ is a random variable $\mathcal{A}_0$-measurable and, for all $j\in J$,
$\phi^V_j$ is a
predictable process satisfying $E^P\left[\int_0^\infty\,(\phi^V_j(s))^2 d[M^j](s)\right]<\infty$.
\end{definition}
\begin{remark}\label{rem:prp basis}
Corollary 11.4 and subsequent remark, p.~340 in \cite{jacod} joint with Theorem 13.4, p.~363 in \cite{he-wang-yan92} imply that if  $(M^j)_{j\in J}$ is an $\mathbb{A}$-basis, then all $\mathbb{A}$-local martingales with null initial value are obtained as sum of stochastic integrals of predictable processes w.r.t.~the elements of the family, that is $(M^j)_{j\in J}$ has the (strong) predictable representation property w.r.t.~$\mathbb{A}$ ($\mathbb{A}$-PRP) (see Definition 13.1, p.~362 in \cite{he-wang-yan92}).
\end{remark}
We call \textit{nontrivial} an $\overline{\mathbb{R}}^+$-valued random time $\eta$, such that $P(\eta<\infty)>0$.\bigskip\\
Given a nontrivial random time $\eta$, the filtration $\sigma(\eta\wedge \cdot)$ is the minimal filtration under the usual conditions that makes $\eta$ a stopping time.~We denote by
$\mathbb{A}^\eta$ the  \textit{progressive enlargement of $\mathbb{A}$ with $\eta$}, that is the filtration defined by
$$\mathcal{A}^\eta_t:=\bigcap_{s>t}\mathcal{A}_s\vee\sigma(\eta\wedge s).$$
In case when $\eta$ is an $\mathbb{A}$-stopping time then it satisfies
 \begin{equation*}
 \eta=\eta^{a, \mathbb{A}}\wedge\eta^{i, \mathbb{A}},
 \end{equation*}
where $\eta^{a, \mathbb{A}}$ and $\eta^{i, \mathbb{A}}$ are the \textit{$\mathbb{A}$-accessible component of $\eta$} and the \textit{$\mathbb{A}$-totally inaccessible component of $\eta$}, respectively (see e.g.~Theorem 1.42,
 p.~19 in \cite{aksamit-jean-17}).~More precisely,
there exist two disjoint events $A$ and $B$ such that  $P$-a.s.
$$A\cup B=(\eta<\infty),$$
 and
 \begin{equation*}
   \eta^{a, \mathbb{A}}=\eta\, \mathbb{I}_A + \infty\, \mathbb{I}_{A^c},\;\;\;\eta^{i, \mathbb{A}}=\eta\, \mathbb{I}_B + \infty\, \mathbb{I}_{B^c}.
 \end{equation*}
Therefore,
 $$\eta\; \mathbb{I}_{\eta<\infty}=\eta^{a, \mathbb{A}}\, \mathbb{I}_A + \eta^{i, \mathbb{A}}\, \mathbb{I}_B.$$
It is worthwhile to recall that, if $T$ is any $\mathbb{A}$-predictable stopping time, then by definition
 \begin{equation*}
 P(\eta^{i, \mathbb{F}}=T<\infty)=0,
 \end{equation*}
As far as $\eta^{a, \mathbb{A}}$ is concerned,
 \begin{equation}\label{eq: accessible-r-t}
 [[\eta^{a, \mathbb{A}}]]\subset\bigcup_{m \ge 1} [[\eta^{a, \mathbb{A}}_m]],
 \end{equation}
where $(\eta^{a, \mathbb{A}}_m)_{m \ge 1} $ is a sequence of $\mathbb{A}$-predictable stopping times.~Such a sequence is not unique and it may be chosen in such a way that the corresponding graphs are pairwise disjoint (see Theorem 3.31, p.~95 in \cite{he-wang-yan92}).~Any sequence
of $\mathbb{A}$-predictable stopping times  $(\eta^{a, \mathbb{A}}_m)_{m \ge 1} $ satisfying (\ref{eq: accessible-r-t}) and with pairwise disjoint graphs takes the name of \textit{enveloping sequence of $\eta^{a, \mathbb{A}}$}.~Clearly, a random time $\eta$ is \textit{$\mathbb{A}$-accessible} when $P(\eta=\eta^{a,\mathbb{A}})=1$, and $\eta$ is  \textit{$\mathbb{A}$-totally inaccessible} when $P(\eta=\eta^{i,\mathbb{A}})=1$.
\begin{remark}
Note that if  $\mathbb{B}$ is a filtration different from $\mathbb{A}$ but such that $\eta$ is also a $\mathbb{B}$-stopping time, then the $\mathbb{B}$-accessible component $\eta^{a, \mathbb{B}}$ is in general different from $\eta^{a, \mathbb{A}}$ and $\eta^{i, \mathbb{B}}$ is in general different from $\eta^{i, \mathbb{A}}$.~Accessibility (totally inaccessibility) may get lost by restriction (expansion) of the filtration, that is an accessible (totally inaccessible) stopping time may be totally inaccessible (accessible)  w.r.t.~a smaller (bigger) filtration.
\end{remark}
\noindent We refer to $\eta^{a,\sigma(\eta\wedge\cdot)}$ as the \textit{naturally accessible component of $\eta$}, and to $\eta^{i,\sigma(\eta\wedge\cdot)}$ as the \textit{naturally totally inaccessible component of $\eta$}.~We call $\eta$ \textit{naturally accessible} when it coincides with $\eta^{a,\sigma(\eta\wedge\cdot)}$,  and \textit{naturally totally inaccessible} when it coincides with $\eta^{i,\sigma(\eta\wedge\cdot)}$.
\begin{remark}\label{rem:predictability}
A nontrivial naturally accessible stopping time cannot be predictable (unless it coincides with a positive constant) and has an atomic law on $[0,\infty]$.~The law of a naturally totally inaccessible stopping time, instead, is diffusive (see Theorem IV-107, p.~241 in \cite{del-me-a}).
\end{remark}
\noindent For any $\mathbb{A}$-stopping time $\eta$, $H^{\eta,\mathbb{A}}$ stays for  the \textit{$\mathbb{A}$-compensated occurrence process of $\eta$}, that is the $\mathbb{A}$-martingale defined by
\begin{equation*}
H^{\eta,\mathbb{A}}:=\mathbb{I}_{\{\eta\leq\cdot\}}-B^{\eta,\mathbb{A}},
\end{equation*}
where $B^{\eta,\mathbb{A}}$ is the $\mathbb{A}$-predictable compensator of $\mathbb{I}_{\{\eta\leq\cdot\}}$ (briefly, \textit{$\mathbb{A}$-compensator} of $\eta$).
\begin{remark}\label{rem pred}
(i) If $\eta$ is $\mathbb{A}$-predictable, then trivially $B^{\eta,\mathbb{A}}\equiv \mathbb{I}_{\{\eta\leq\cdot\}}$ so that $H^{\eta,\mathbb{A}}\equiv 0$.\\
(ii) $B^{\eta,\mathbb{A}^\eta}$ is a process stopped at $\eta$, and moreover there exists an $\mathbb{A}$-predictable process, $\Lambda^{\eta,\mathbb{A}}$, called the \textit{$\mathbb{A}$-reduction  of the compensator of $\eta$}, such that  \begin{equation}\label{eq:JY-formula}B^{\eta,\mathbb{A}^\eta}=\Lambda^{\eta,\mathbb{A}}_{\eta\wedge\cdot}\end{equation} (see Proposition 2.15, p.37 in \cite{aksamit-jean-17}).
\end{remark}
We now introduce the notion of \textit{avoidance} (see Definition 1.47, p.22 in \cite{aksamit-jean-17}).
\begin{definition}\label{def:avoidance}
A random time $\eta$ avoids $\mathbb{A}$-stopping times if $P(\eta = T < \infty) = 0$ for every $\mathbb{A}$-stopping time $T$.
\end{definition}
We refer to the above property as Hypothesis $(\mathcal{A})$ w.r.t.~$\mathbb{A}$ .
\begin{remark}\label{rem:avoidance tot.in.ty}
If $\eta$ satisfies Hypothesis $(\mathcal{A})$ w.r.t.~$\mathbb{A}$, then the $\mathbb{A}^\eta$-compensator of $\eta$ is continuous (see Lemma 1.48, p.~22 in \cite{aksamit-jean-17}), and hence $\eta$ is $\mathbb{A}^\eta$-totally inaccessible (see Theorem 1.43, p.~20 in \cite{aksamit-jean-17}).
\end{remark}
The following theorem is the key point of our work.
\begin{thm}\label{thm:ttdecom}({Theorem 2.5 in \cite{ak-chou-jean-18}})
Any random time $\eta$ in $[0,\infty]$ has an \textit{$\mathbb{A}$-thin-thick decomposition}, that is there exists a pair of random times  $(\eta_1,\eta_2)$ such that
 \begin{equation*}
 \eta=\eta_1\wedge\eta_2,\;\;\;\;\eta_1\vee\eta_2=\infty,
 \end{equation*}
where
$$[[\eta_1]]\subset \bigcup_{n=1}^{\infty} [[T_n]]$$ with $(T_n)_{n\geq 1},$ a (possibly finite) sequence  of $\mathbb{A}$-stopping times with disjoint graphs, and $\eta_2$ satisfies Hypothesis $(\mathcal{A})$ w.r.t.~$\mathbb{A}$.~This decomposition is a.s.~unique.
\end{thm}
The random times $\eta_1$  and $\eta_2$ are called the \textit{$\mathbb{A}$-thin part} of $\eta$ and the \textit{$\mathbb{A}$-thick part} of $\eta$, respectively.
The sequence $(T_n)_{n\geq 1}$ is called an \textit{$\mathbb{A}$-exhausting sequence of $\eta_1$}.~Set, for all $n\geq 1$,
\begin{equation*} C_n:=\left\{\eta_1=T_n<\infty\right\}.\end{equation*}
Then trivially $C_n=\left\{\eta=T_n<\infty\right\}$, any event $C_n$ belongs to $\mathcal{A}^{\eta}_\infty$, and $\left\{(C_n)_{n\geq 1}, \{\eta_1=\infty\}\right\}$ is a partition of $\Omega$.
\begin{remark}\label{rem:increasing}
W.l.o.g.~the sequence $(T_n)_{n\ge 1}$ can be chosen a.s.~increasing (see Theorem 2.5 in \cite{ak-chou-jean-18}).
\end{remark}
\begin{remark}
Immediate consequences of previous definitions are the following:\\if $\eta_1$ is nontrivial, then $\eta$ does not satisfy Hypothesis $(\mathcal{A})$;\\if $\eta$ is an $\mathbb{A}$-stopping time, then  $\eta=\eta_1$ a.s.~and $\eta_2$ is trivial;\\if $\eta_1$ is an $\mathbb{A}$-stopping time, then either $\eta_2$ is trivial and $\eta=\eta_1$ a.s.,~or $\eta_2$ is  nontrivial and  $P(\eta_1<\infty)<1$.
\end{remark}
By Theorem 5.1 in \cite{ak-chou-jean-18} we get next result.
\begin{proposition}\label{rem:composition} $\mathbb{A}^{\eta}=\big(\mathbb{A}^{\eta_1}\big)^{\eta_2}=\big(\mathbb{A}^{\eta_2}\big)^{\eta_1}$.
\end{proposition}
Since $\sigma(\eta_1\wedge\cdot)$ is contained in the natural filtration of the multivariate point process $(T_n, \mathbb{I}_{C_n}, n\geq 1)$ (see Remark \ref{rem:increasing}), Theorem T30, p.307 in \cite{brem-81} applies to prove the following proposition.
\begin{proposition}\label{sigma}
(a) For all $n\geq 1$, $\mathcal{A}^{\eta_1}_{T_n}=\mathcal{A}_{T_n}\vee\sigma(C_1,...,C_n)$;\\
(b) $\mathcal{A}^{\eta_1}_{T^-_1}=\mathcal{A}_{T^-_1}$, and for all $n\geq 2$, $\mathcal{A}^{\eta_1}_{T^-_n}=\mathcal{A}_{T^-_n}\vee \sigma(C_1,...,C_{n-1})$.
\end{proposition}
Finally, we recall the definition of the so-called \textit{immersion property} which plays an important role in our paper (see Definition 1.20, p.~11 in \cite{aksamit-jean-17}).
\begin{definition}
 Let $\mathbb{A}$ and $\mathbb{B}$ be two filtrations on $(\Omega, \mathcal{F}, P)$ such that $\mathbb{A}\subset\mathbb{B}$.~Then $\mathbb{A}$ is immersed in $\mathbb{B}$ ($\mathbb{A}\hookrightarrow\mathbb{B}$) if any $\mathbb{A}$-local martingale is also a $\mathbb{B}$-local martingale.
\end{definition}
\section{Hypotheses $(\mathcal{P})$ and $(\mathcal{C})$}\label{sec3}
In this section we present the two hypotheses that characterize the class of random times we will consider in our main result.\bigskip\\
With the following definition, new to the literature, we aim to capture the random times that do not satisfy Hypothesis $(\mathcal{A})$, but overlap only with $\mathbb{A}$-predictable stopping times.
\begin{definition}\label{def:P}
A random time $\eta$  with nontrivial $\mathbb{A}$-thin part satisfies Hypothesis ($\bold{\mathcal{P}}$) w.r.t.~$\mathbb{A}$ if $\eta_1$ has an $\mathbb{A}$-exhausting sequence $(T_n)_{n\ge 1}$ of $\mathbb{A}$-predictable stopping times.
\end{definition}
Some examples of random times that satisfy Hypothesis ($\bold{\mathcal{P}}$) can be found in Section \ref{sec5}.~We discuss some consequences of this hypothesis.
\begin{proposition}\label{prop:1}
Let $\eta$ be a random time with nontrivial $\mathbb{A}$-thin part satisfying Hypothesis $(\mathcal{P})$ w.r.t.~$\mathbb{A}$.~Then:\\
(i)
$\eta_1$ is an $\mathbb{A}^{\eta_1}$-accessible stopping time with enveloping sequence $(T_n)_{n\ge 1}$;\\
(ii)  $\eta_1$ and $\eta_2$ are the $\mathbb{A}^{\eta}$-accessible component and the  $\mathbb{A}^{\eta}$-totally inaccessible component of $\eta$, respectively, that is
$\eta_1=\eta^{a, \mathbb{A}^{\eta}}$ and $\eta_2=\eta^{i, \mathbb{A}^{\eta}}$.
\end{proposition}
\begin{proof}
(i) It follows immediately considering that any $\mathbb{A}$-predictable stopping time is also $\mathbb{A}^{\eta_1}$-predictable.\\
(ii) Previous point implies that $\eta_1$ is $\mathbb{A}^{\eta}$-accessible, since $\mathbb{A}^{\eta_1}\subset \mathbb{A}^{\eta}$.~Moreover, $\eta_2$ satisfies Hypothesis $(\mathcal{A})$ w.r.t.~$\mathbb{A}^{\eta_1}$, since $\eta_2$ avoids $\mathbb{A}$-stopping times.~By Remark \ref{rem:avoidance tot.in.ty}
 $\eta_2$ is $\big(\mathbb{A}^{\eta_1}\big)^{\eta_2}$-totally inaccessible, that is $\eta_2$ is $\mathbb{A}^{\eta}$-totally inaccessible (see Proposition  \ref{rem:composition}).
\end{proof}
\begin{remark}\label{rem:trivial}
Observe that $\eta$ with nontrivial $\mathbb{A}$-thin part satisfies Hypothesis $(\mathcal{P})$ w.r.t.~$\mathbb{A}$ with $T_n=t_n\in\mathbb{R}^+, \; n\geq 1,$ if and only if $\eta_1$ is naturally accessible.~In fact: $\eta_1$ is naturally accessible if and only if it is a discrete random variable on $[0,\infty]$ (see Remark \ref{rem:predictability}); any deterministic time is predictable w.r.t.~any filtration.\\
Moreover, when $\eta_1$ is naturally accessible, the (possibly finite) ordered sequence of atoms in $[0, \infty)$ of the law of $\eta_1$ is both an  $\mathbb{A}$-enveloping sequence and an $\mathbb{A}$-exhausting sequence of $\eta_1$.
\end{remark}
\begin{proposition}\label{lemma:eq:Delta}
Let $\eta$ be a random time with nontrivial $\mathbb{A}$-thin part satisfying Hypothesis $(\mathcal{P})$ w.r.t.~$\mathbb{A}$.~Then:\bigskip\\
(i)
\begin{equation}\label{eq-comp222b}
H^{\eta_1, \mathbb{A}^{\eta_1}}=\mathbb{I}_{\{\eta_1\le
   \cdot\}}-\sum_{n\ge 1}P(C_n\mid\mathcal{A}^{\eta_1}_{T_n^-})\mathbb{
I}_{\{T_n\le \cdot\}};
\end{equation}
(ii) $\{\Delta H^{\eta_1,\mathbb{A}^{\eta_1}}\neq 0\}$ is an $\mathbb{A}$-thin set with $(T_n)_{n\ge 1}$ as exhausting sequence, that is
\begin{equation}\label{eq:delta}\{\Delta H^{\eta_1,\mathbb{A}^{\eta_1}}\neq 0\}\subset\bigcup_{n \ge 1}\, [[T_n]].\end{equation}
\end{proposition}
\begin{proof}
(i) It is a direct consequence of Lemma 2.11 in \cite{caltor24} and point (i) of Proposition \ref{prop:1}.\bigskip\\
(ii) Since  $$\mathbb{I}_{\{\eta_1\le
   \cdot\}}=\sum_{n\ge 1}\mathbb{I}_{C_n}\mathbb{I}_{\{\eta_1\le
   \cdot\}},$$
equality (\ref{eq-comp222b}) can be rewritten as
\begin{equation}\label{eq-comp222}
H^{\eta_1, \mathbb{A}^{\eta_1}}=\sum_{n\ge 1}\big(\mathbb{I}_{C_n}-P(C_n\mid\mathcal{A}^{\eta_1}_{T_n^-})\big)\mathbb{
I}_{\{T_n\le
  \cdot\}},
\end{equation}
and relation (\ref{eq:delta}) follows immediately.
\end{proof}
\begin{remark}
By (\ref{eq-comp222b}) we derive the following representation for the $\mathbb{A}^{\eta_1}$-compensator of $\eta_1$.
\begin{equation}\label{eq-comp222h}
B^{\eta_1, \mathbb{A}^{\eta_1}}=\sum_{n\ge 1}P(C_n\mid\mathcal{A}^{\eta_1}_{T_n^-})\mathbb{
I}_{\{T_n\le \cdot\}}
\end{equation}
Let us observe that
$$P(C_n\mid \mathcal{A}^{\eta_1}_{T_n^-})=P(C_n\cap\{\eta_1\geq T_n\}\mid \mathcal{A}^{\eta_1}_{T_n^-})=P(C_n\mid \mathcal{A}^{\eta_1}_{T_n^-})\mathbb{I}_{\{\eta_1\geq T_n\}},$$
where last equality follows by $$\{\eta_1 < T_n\} \in \mathcal{A}^{\eta_1}_{T_n^-}.$$
Then it also holds
\begin{equation}\label{eq-comp222hg}
B^{\eta_1, \mathbb{A}^{\eta_1}}=\sum_{n\ge 1}P(C_n\mid\mathcal{A}^{\eta_1}_{T_n^-})\mathbb{
I}_{\{T_n\le \eta_1\wedge\cdot\}},
\end{equation}
(see Remark \ref{rem pred} (ii)).~A different way to derive (\ref{eq-comp222hg}) is based on the following observation.~Fixed $m\ge 1$, for all $n>m$ it holds $C_m\in\mathcal{A}^{\eta_1}_{T_n^-}$.~Thus
$$\mathbb{I}_{C_m}P(C_n\mid\mathcal{A}^{\eta_1}_{T_n^-})=P(C_m \cap C_n\mid\mathcal{A}^{\eta_1}_{T_n^-})=0,$$
where last equality follows considering that  $C_m$ and  $C_n$ for $n\neq m$ are disjoint sets.
~It is to note that (\ref{eq-comp222hg}) differs from the classical expression based on the reduction compensator of  $\eta_1$, $\Lambda^{\eta_1,\mathbb{A }}$ (see (\ref{eq:JY-formula})), since  the right hand side of (\ref{eq-comp222h}) is not an $\mathbb{A}$-predictable process.
\end{remark}
\begin{proposition}\label{prop:00}
If all $\mathbb{A}$-local martingales are continuous, then any random time with nontrivial $\mathbb{A}$-thin part satisfies Hypothesis $(\mathcal{P})$ w.r.t.~$\mathbb{A}$.
\end{proposition}
\begin{proof}
If all $\mathbb{A}$-local martingales are continuous, then they are also quasi-left-continuous,
so that $\mathbb{A}$ is a quasi-left-continuous filtration (see  Proposition 1.52, p.~22 in \cite{jacod}).~This property
in turns implies that any $\mathbb{A}$-accessible stopping time is  $\mathbb{A}$-predictable (see Theorem IV-83, p.~217 in \cite{del-me-a}).~Let us suppose
now that there exists a nontrivial $\mathbb{A}$-totally inaccessible stopping time $T$.~Then, it admits a continuous  $\mathbb{A}$-compensator $B^{T, \mathbb{A}}$ (see Theorem VI-78, p.~151 in \cite{del-me-b}) so
that $\mathbb{I}_{\{T\le \cdot\}}-B^{T, \mathbb{A}}$ is an $\mathbb{A}$-martingale with a jump, and this contradicts the hypothesis.~Thus, all $\mathbb{A}$-stopping times are predictable.
\end{proof}
\begin{cor}\label{prop:0}
If  $\mathbb{A}$ is a Brownian filtration, then any random time with nontrivial  $\mathbb{A}$-thin part satisfies Hypothesis $(\mathcal{P})$ w.r.t.~$\mathbb{A}$.
\end{cor}
\vspace{.5em}
We now introduce the second main hypothesis of our paper.~Roughly speaking it deals with the continuity condition of the filtration $\mathbb{A}$ on the exhausting sequence  of $\eta_1$.
\begin{definition}\label{def:C}
A random time $\eta$  with nontrivial $\mathbb{A}$-thin part $\eta_1$ satisfies Hypothesis ($\bold{\mathcal{C}}$) w.r.t.~$\mathbb{A}$ if, for any $n\geq 1$, $\mathcal{A}_{T_n}=\mathcal{A}_{T_n^-}$, where $({T_n})_{n\ge 1}$ is an $\mathbb{A}$-exhausting sequence of $\eta_1$.
\end{definition}
To illustrate better this definition let us state an auxiliary result.
\begin{proposition}\label{prop:auxiliary}
Let $(S_n)_{n\ge 1}$ be a sequence of  $\mathbb{A}$-stopping times. Then
the following conditions are equivalent
\begin{enumerate}
 \item  $(S_n)_{n\ge 1}$ are $\mathbb{A}$-predictable stopping times and Hypothesis ($\bold{\mathcal{C}}$) holds true;
  \item any $\mathbb{A}$-local martingale does'nt jump on the random times  $(S_n)_{n\ge 1}$.
\end{enumerate}
\end{proposition}
\begin{proof}
The direct part derives by noting that if $Z$ is an $\mathbb{A}$-local martingale then for any $n\in \mathbb{N}$, since $S_n$ is predictable, $E[\Delta Z_{S_n}\mid\mathcal{A}_{S_n^-}]=0$ (see Lemma 1.21, p. 13 in  \cite{jacod}).~Moreover, Hypothesis ($\bold{\mathcal{C}}$) implies that $\Delta Z_{S_n}$ is $\mathcal{A}_{S_n^-}$-measurable and therefore
$\Delta Z_{S_n}=E[\Delta Z_{S_n}\mid\mathcal{A}_{S_n^-}]=0$, $P$-a.s.~The viceversa works as follows.~First of all each $S_n$ cannot neither totally inaccessible, nor accessible since, in both cases, the compensated occurrence process of $S_n$ would be a nontrivial  $\mathbb{A}$-local martingale jumping on $S_n$.~Then $S_n$ can only be predictable.~Moreover if there would exist $n\in \mathbb{N}$ such that
$\mathcal{A}_{S_n}\neq\mathcal{A}_{S_n^-}$ then it would be possible to find an  event $B$ such that  $B\in \mathcal{A}_{S_n}$ but $B\not\in \mathcal{A}_{S_n^-}$.~So, considering the   $\mathbb{A}$-stopping time $T_B:=\mathbb{I}_BS_n+\infty\mathbb{I}_{A^c}$ (see Theorem 3.9,  p.85 in \cite{he-wang-yan92}), the compensated occurrence process
$$\mathbb{I}_{T_B\le \cdot}-P(B|\mathcal{A}_{S_n^-})\mathbb{I}_{S_n\le \cdot} =\left(\mathbb{I}_B-P(B|\mathcal{A}_{S_n^-})\right)\mathbb{I}_{S_n\le \cdot}$$
would be an $\mathbb{A}$-local martingale jumping on $S_n$, in contradiction with the hypothesis.
\end{proof}
In light of previous proposition, \textit{Hypotheses $(\mathcal{P})$ and ($\bold{\mathcal{C}}$) are equivalent to avoidance for $\eta$ of the jumps of any $\mathbb{A}$-local martingale}.~Thus, by recalling that the local martingales w.r.t.~a quasi-left continuous filtration admit totally inaccessible jumps only we get immediately  next result.
\begin{proposition}\label{prop:qlc}
 Let $\mathbb{A}$ be a quasi-left continuous filtration and $\eta$ a random time with nontrivial thin part.~If $\eta$  satisfies Hypothesis $(\mathcal{P})$  w.r.t.~$\mathbb{A}$ then $\eta$  satisfies Hypothesis ($\bold{\mathcal{C}}$) w.r.t.~$\mathbb{A}$.\end{proposition}
\begin{remark}\label{ex:1}
Let us show by a counterexample that Hypothesis  $(\mathcal{C})$ is not always satisfied by all pairs ($\mathbb{A}$, $\eta$), with $\mathbb{A}$ a quasi-left continuous filtration.~On a probability space $(\Omega, \mathcal{F}, P)$ let us consider a pair of real-valued independent random variables $(\eta_1,Z_1)$ where  $\eta_1$ has  exponential law.~Let  $\mathbb{A}$ be the filtration generated by the process $Z_1\mathbb{I}_{\eta_1\le \cdot~}$Note that $\eta_1$ is a naturally totally inaccessible random time (see Remark \ref{rem:predictability}) and, since the independence assumption, $\eta_1$ is also  totally inaccessible w.r.t.~$\mathbb{A}$.~Thus,  $\mathbb{A}$ is  a quasi-left continuous filtration.~In fact, although $\mathcal{A}_{\eta_1^-}=\sigma(\eta_1)\neq\mathcal{A}_{\eta_1}=\sigma(\eta_1,Z_1)$  (see Theorem T30, p.~307 in \cite{brem-81}), instead  $\mathcal{A}_{S^-}=\mathcal{A}_{S}$ for any $\mathbb{A}$-predictable stopping time $S$.~Now,  set $\eta:=\eta_1\wedge \eta_2$, with $\eta_2$ an   $\mathbb{A}$-avoiding stopping time.~Then, $\eta$ does'nt satisfy Hypothesis $(\mathcal{C})$ w.r.t.~$\mathbb{A}$.
\end{remark}
\begin{example}\label{ex:0} We provide a simple example in which $\mathbb{A}$ is not a quasi-left continuous filtration, nevertheless $\eta$ satisfies both $(\mathcal{P})$ and $(\mathcal{C})$.~Let $T$ be a random time assuming values in
$\{1,2,3\}$ and
$\mathbb{A}:=\sigma(T\wedge\cdot).$
Then $\mathbb{A}$ is not a quasi-left continuous filtration.~In fact,
$$\mathcal{A}_t=
\begin{cases}
&\{\phi, \Omega\}  \  \text{ if } t<1;\\
&\sigma(\{T=1\})  \  \text{ if } t\in[1,2);\\
&\sigma(T)   \  \text{ if } t\ge 2,
\end{cases}
$$
 so that, $\{T=2\}\in \mathcal{A}_2$, but $\{T=2\}\not\in \mathcal{A}_{2^-}$.~Consider a second random time $\eta$ with values in  $\{4,5,6\}$.~Now $\eta$, being a discrete random time, satisfies Hypothesis $(\mathcal{P})$ w.r.t.~any filtration and in particular w.r.t.~$\mathbb{A}$ (see Remark \ref{rem:trivial}).~Moreover, it is immediate to check that $\eta$ satisfies also Hypothesis $(\mathcal{C})$ w.r.t.~$\mathbb{A}$.\end{example}
Let us present now a broad class of filtrations w.r.t.~which every random time with nontrivial thin part satisfies $(\mathcal{P})$ and $(\mathcal{C})$.
\begin{proposition}\label{cor2}
Let $\mathbb{A}$ be such that all $\mathbb{A}$-local martingales are continuous and let $\eta$ be a random time with nontrivial $\mathbb{A}$-thin part.~Then, $\eta$ satisfies both Hypotheses $(\mathcal{P})$ and $(\mathcal{C})$ w.r.t.~$\mathbb{A}$.
\end{proposition}
\begin{proof}
It follows immediately by Proposition \ref{prop:auxiliary}.
\end{proof}
\section{Martingale representations}\label{sec4}
\subsection{The setting and the $\mathbb{F}^\tau$-compensator of $\tau$}
Let's introduce the setting under which we will work throughout this section.\bigskip\\
\textit{On a probability space $(\Omega,\mathbb{F},P)$, $\mathbb{F}$ is a  filtration under the usual conditions such that $\mathcal{F}_0$ is trivial, $(M^j)_{j\in J}$, with $J$ a finite subset of $\mathbb{N}$, is an $\mathbb{F}$-basis, and $\tau$  is a random time with nontrivial $\mathbb{F}$-thin part.}\bigskip\\
In this environment, under immersion the $\mathbb{F}^{\tau}$-compensator of $\tau$ takes the  representation  (\ref{eq-comp}) fundamental in the proof of our main result.~For notational convenience, we still denote by $(T_n)_{n\ge 1}$ an $\mathbb{F}$-exhausting sequence of $\tau_1$ and by $(C_n)_{n\ge 1}$ the sequence of sets defined by  $C_n:=\left\{\tau_1=T_n<\infty\right\}$.
\begin{proposition}\label{th:compen}
Assume that $\tau$ satisfies Hypothesis $(\mathcal{P})$ w.r.t.~$\mathbb{F}$.~If $\mathbb{F}\hookrightarrow\mathbb{F}^{\tau}$, then
\begin{equation}\label{eq-comp225}
P(C_n\mid\mathcal{F}^{\tau_1}_{T_n^-})=P(C_n\mid\mathcal{F}^{\tau}_{T_n^-}),
\end{equation}
and
\begin{equation}\label{eq:generalized}
H^{\tau,\mathbb{F}^\tau}=\mathbb{I}_{\{\tau\leq\cdot\}}-\sum_{n\ge 1}P(C_n\mid\mathcal{F}^{\tau}_{T_n^-})\mathbb{I}_{\{T_n\le \cdot\}}-B^{\tau_2,\mathbb{F}^{\tau}}.
\end{equation}
\end{proposition}
\begin{proof}
By Proposition 5.4 in \cite{ak-chou-jean-18} $\mathbb{F}^{\tau_1}\hookrightarrow\mathbb{F}^\tau$ and $\mathbb{F}^{\tau_2}\hookrightarrow\mathbb{F}^\tau$, and therefore by the uniqueness of the $\mathbb{F}^\tau$-compensator of $\tau_1$ and of the $\mathbb{F}^\tau$-compensator of $\tau_2$ (see Theorem VI-80, p.139, in \cite{del-me-b}) we get
\begin{equation*}B^{\tau_1,\mathbb{F}^{\tau}} = B^{\tau_1,\mathbb{F}^{\tau_1}},\;\;\;\;B^{\tau_2,\mathbb{F}^{\tau}} = B^{\tau_2,\mathbb{F}^{\tau_2}}.\end{equation*}
By the first of the above equalities and Lemma 2.11 in \cite{caltor24} the equality (\ref{eq-comp225}) follows. Moreover
\begin{equation}\label{eq:doppia}H^{\tau_1,\mathbb{F}^{\tau}} = H^{\tau_1,\mathbb{F}^{\tau_1}},\;\;\;H^{\tau_2,\mathbb{F}^{\tau}}=H^{\tau_2,\mathbb{F}^{\tau_2}}.\end{equation}
To show (\ref{eq:generalized}) by the following simple equality valid for all $t\ge 0$
\begin{equation*}
\mathbb{I}_{\{\tau\le t\}}=\mathbb{I}_{\{\tau=\tau_{1}<\infty\}}\mathbb{I}_{\{\tau_1\le t\}}+\mathbb{I}_{\{\tau=\tau_{2}<\infty\}}\mathbb{I}_{\{\tau_{2}\le t\}},
\end{equation*}
and observing that
$$\{\tau_{1}\le t\}\subset \{\tau=\tau_{1}<\infty\}, \ \ \{\tau_{2}\le t\}\subset \{\tau=\tau_{2}<\infty\},$$
we get
\begin{equation}\label{eq-dectau}
\mathbb{I}_{\{\tau\le t\}}=\mathbb{I}_{\{\tau_{1}\le t\}}+\mathbb{I}_{\{\tau_{2}\le t\}}.
\end{equation}
Therefore we derive easily
 \begin{equation}\label{eq:sumH}
 H^{\tau,\mathbb{F}^\tau}= H^{\tau_1,\mathbb{F}^\tau} + H^{\tau_2,\mathbb{F}^\tau}.
 \end{equation}
Thus, (\ref{eq:generalized}) follows by (\ref{eq:sumH}) and (\ref{eq:doppia}) recalling (\ref{eq-comp222b}) and (\ref{eq-comp225}), and finally using (\ref{eq-dectau}).
\end{proof}
\begin{remark}\label{rem:remcomp}
By (\ref{eq:generalized}) we derive the following representation for the $\mathbb{F}^{\tau}$-compensator of $\tau$:
\begin{equation*}
B^{\tau,\mathbb{F}^\tau}=\sum_{n\ge 1}P(C_n\mid\mathcal{F}^{\tau}_{T_n^-})\mathbb{I}_{\{T_n\le \cdot\}}+B^{\tau_2,\mathbb{F}^{\tau}}.
\end{equation*}
Recalling Remark \ref{rem pred} (ii), we can rewrite
\begin{equation}\label{eq:comp}
B^{\tau,\mathbb{F}^\tau}=\sum_{n\ge 1}P(C_n\mid\mathcal{F}^{\tau}_{T_n^-})\mathbb{I}_{\{T_n\le \tau \wedge \cdot\}}+\Lambda^{\tau_2,\mathbb{F}^{\tau}}_{\tau \wedge \cdot},
\end{equation}
where $\Lambda^{\tau_2,\mathbb{F}^{\tau}}$ is the $\mathbb{F}^{\tau}$-reduction  of the compensator of $\tau_2$.~Let us observe that
if $\tau_2$ is trivial, then $\tau=\tau_1$, a.s.~and in the right hand side of the above expression the term $\Lambda^{\tau_2,\mathbb{F}^{\tau}}_{\tau \wedge \cdot}$ is identically null.
\end{remark}
\subsection{The first step: an $\mathbb{F}^{\tau_1}$-basis}
With the next result we take the first step towards building an $\mathbb{F}^{\tau}$-basis.
\begin{thm}\label{cor}
Assume Hypotheses $(\mathcal{P})$ and $(\mathcal{C})$ for $\tau$ w.r.t.~$\mathbb{F}$.~If $\mathbb{F}\hookrightarrow\mathbb{F}^{\tau_1}$, then:\\
(i) for all $j\in J$,
\begin{equation}\label{eq01}\big[H^{\tau_1,\mathbb{F}^{\tau_1}}, M^j\big]\equiv 0\;\;\;a.s.,\end{equation}
so that $M^j$ and $H^{\tau_1,\mathbb{F}^{\tau_1}}$ are $\mathbb{F}^{\tau_1}$-orthogonal martingales;\\
(ii) the family $\left\{(M^j)_{j\in J},H^{\tau_1, \mathbb{F}^{\tau_1}}\right\}$ is an $\mathbb{F}^{\tau_1}$-basis.
\end{thm}
\begin{proof}
  (i)  By the immersion property of $\mathbb{F}$ in $\mathbb{F}^{\tau_1}$ all elements of the family $\big\{(M^j)_{j\in J}, H^{\tau_1,\mathbb{F}^{\tau_1}}\big\}$ are  $\mathbb{F}^{\tau_1}$-martingales, and all covariation processes in (\ref{eq01}) are well-defined.~Moreover, since $H^{\tau_1,\mathbb{F}^{\tau_1}}$ is a process of finite variation,
\begin{equation*}\big[H^{\tau_1,\mathbb{F}^{\tau_1}},M^j\big]=\sum_{s\leq \cdot}\Delta H^{\tau_1,\mathbb{F}^{\tau_1}}_s\,\Delta M^j_s.\end{equation*}
  Hypothesis $(\mathcal{C})$ implies that any martingale $M^j, j\in J,$ cannot jump at any time $T_n, n\geq 1$.~Hence, (\ref{eq01}) derives immediately from point (ii) of Proposition \ref{lemma:eq:Delta}.\bigskip\\
(ii) Let us sketch the line of the proof.~If $\tau_1$ is an $\mathbb{F}$-stopping time, then: $\mathbb{F}^{\tau_1}=\mathbb{F}$ and the statement is trivial since $H^{\tau_1, \mathbb{F}^{\tau_1}}=0$  (see Remark \ref{rem pred}).~Otherwise, we proceed by using the machinery of the II Theorem of Asset Pricing: we show that the set of probability measure on $(\Omega, \mathcal{F}^{\tau_1}_\infty)$ equivalent to $P$ under which  $\left\{(M^j)_{j\in J},H^{\tau_1, \mathbb{F}^{\tau_1}}\right\}$ is a family of $\mathbb{F}^{\tau_1}$-local martingales coincides with $\{P\}$, so that it turns out that $P$ is an extremal point of the space of probability measures making $\left\{(M^j)_{j\in J},H^{\tau_1, \mathbb{F}^{\tau_1}}\right\}$ a family of $\mathbb{F}^{\tau_1}$-local martingales; since by point (i) the martingales of the family are pairwise orthogonal,  Corollary 11.4 and subsequent remark, p.~340 in \cite{jacod} end the proof.\bigskip\\
Let $Q$ be a probability measure on $(\Omega, \mathcal{F}^{\tau_1}_\infty)$ equivalent to $P$ such that the set $\left\{(M^j)_{j\in J},
H^{\tau_1, \mathbb{F}^{\tau_1}}\right\}$ is a family of $\mathbb{F}^{\tau_1}$-local martingales.~By assumption $\mathcal{F}_0$ is trivial and  $(M^j)_{j\in J}$ is an $\mathbb{F}$-basis so that using the II Theorem of Asset Pricing we get
\begin{equation}\label{eq:equiva1}
Q|_{\mathcal{F}_\infty}=P|_{\mathcal{F}_\infty}.
\end{equation}
Observe that $(M^j)_{j\in J}\in\mathcal{M}^2(Q, \mathbb{F}^{\tau_1})$ along with $\mathbb{F}\hookrightarrow\mathbb{F}^{\tau_1}$ implies
\begin{equation}\label{eq:equiva2}
\mathbb{F}\hookrightarrow_{Q}\mathbb{F}^{\tau_1},
\end{equation}
that is the immersion property of $\mathbb{F}$ in $\mathbb{F}^{\tau_1}$ under $Q$ (see Section 3.4, pp.~65, 66 in \cite{ak-chou-jean-18}).\bigskip\\
 Since the uniqueness of the compensator of $\tau_1$, $H^{\tau_1, \mathbb{F}^{\tau_1}}$ is an $\mathbb{F}^{\tau_1}$-martingale under $Q$ if and only if for all $n\geq 1$
\begin{equation}\label{eq:equa}
P(C_n\mid\mathcal{F}^{\tau_1}_{T_n^-})=Q(C_n\mid\mathcal{F}^{\tau_1}_{T_n^-}).
\end{equation}
In fact, since $\{T_n=\infty\}\in\mathcal{F}^{\tau_1}_{T_n^-}$
$$P(C_n\mid\mathcal{F}^{\tau_1}_{T_n^-})\mathbb{I}_{\{T_n=\infty\}}=P(C_n\cap\{T_n=\infty\}\mid\mathcal{F}^{\tau_1}_{T_n^-})=0,$$
and analogously $Q(C_n\mid\mathcal{F}^{\tau_1}_{T_n^-})\mathbb{I}_{\{T_n=\infty\}}=0$.~Moreover,
on the set $\{T_n<\infty\}$ by (\ref{eq-comp222}) for all $t\ge 0$
$$P(C_n\mid\mathcal{F}^{\tau_1}_{T_n^-})\mathbb{I}_{\{T_n\le t\}}=Q(C_n\mid\mathcal{F}^{\tau_1}_{T_n^-})\mathbb{I}_{\{T_n\le t\}}.$$
The filtration $\mathbb{F}^{\tau_1}$ satisfies
\begin{equation*}
\mathcal{F}^{\tau_1}_t=\bigcap_{u>t}\mathcal{F}_u\vee\sigma(C_n\cap \{T_n\leq s\}, s\leq u, n\geq 1)
\end{equation*}
(see Lemma 1.5  in \cite{ak-chou-jean-18}), and therefore
\begin{equation}\label{eq:final}
\mathcal{F}^{\tau_1}_{\infty}=\bigvee_ {n\geq 1} \mathcal{F}_\infty\vee\sigma (C_1,\ldots, C_n).
\end{equation}
We apply an iterative procedure to prove that for any $n\geq 1$
\begin{equation}\label{eq:equiva2bis}
Q|_{\mathcal{F}_\infty\vee\sigma (C_1,\ldots, C_n)}=P|_{\mathcal{F}_\infty\vee \sigma (C_1,\ldots, C_n)}.
\end{equation}
Finally, since (\ref{eq:final}), a monotone class argument yields that  $P$ and $Q$ coincide on $\mathcal{F}^{\tau_1}_{\infty}$.\bigskip\\
The first step of the induction is
\begin{equation*}
Q|_{\mathcal{F}_\infty\vee\sigma (C_1)}=P|_{\mathcal{F}_\infty\vee \sigma (C_1)},
\end{equation*}
or equivalently for any $A\in \mathcal{F}_\infty$
\begin{equation}\label{eq:step1}
P(A\cap C_1)=Q(A\cap C_1).
\end{equation}
In fact, fixed $A\in \mathcal{F}_\infty$
\begin{equation}\label{eq:equiva4}
P(A\cap C_1)=E^P\big[P(A\cap C_1\mid \mathcal{F}_{T_1})\big]=E^Q\big[P(A\cap C_1\mid \mathcal{F}_{T_1})\big],
\end{equation}
where the second equality follows by  (\ref{eq:equiva1}) using $\mathcal{F}_{T_1}\subset \mathcal{F}_{\infty}$.\bigskip\\
The immersion property $\mathbb{F}\hookrightarrow\mathbb{F}^{\tau_1}$ implies that, under $P$, $\mathcal{F}_{\infty}$ is conditionally independent of $\mathcal{F}^{\tau_1}_{T_1}$ given $\mathcal{F}_{T_1}$ (see consequence (c.1) of Theorem 3 in \cite{Bre-yor78}).~Now, $A\in \mathcal{F}_{\infty}$ and $C_1\in \mathcal{F}^{\tau_1}_{T_1}$  (see Proposition \ref{sigma} (a)).~Therefore
\begin{equation}\label{eq:equiva5}
E^Q\big[P(A\cap C_1\mid \mathcal{F}_{T_1})\big]=E^Q\big[P(A\mid \mathcal{F}_{T_1})P(C_1\mid \mathcal{F}_{T_1})\big].
\end{equation}
By (\ref{eq:equiva1}), since $A\in \mathcal{F}_\infty$ and $\mathcal{F}_{T_1}\subset \mathcal{F}_\infty$, it follows
\begin{equation}\label{eq: step11}
P(A\mid \mathcal{F}_{T_1})=Q(A\mid \mathcal{F}_{T_1}).\end{equation}
Condition $(\mathcal{C})$ gives $\mathcal{F}_{T_1}=\mathcal{F}_{T_1^-}$ and therefore
\begin{equation}\label{eq: step12}P(C_1\mid \mathcal{F}_{T_1})=P(C_1\mid \mathcal{F}_{T_1^-})=Q(C_1\mid \mathcal{F}_{T_1^-})=Q(C_1\mid \mathcal{F}_{T_1}),\end{equation}
where the middle equality follows by (\ref{eq:equa}) with $n=1$.\bigskip\\
Using (\ref{eq: step12}), (\ref{eq: step11}) and (\ref{eq:equiva5}) equality (\ref{eq:equiva4}) becomes
\begin{equation}\label{eq:equiva6}
P(A\cap C_1)=E^Q\big[Q(A\mid \mathcal{F}_{T_1})Q(C_1\mid \mathcal{F}_{T_1})\big].
\end{equation}
Finally, (\ref{eq:equiva2}) implies that, also under $Q$, $\mathcal{F}_\infty$ is conditionally independent of $\mathcal{F}^{\tau_1}_{T_1}$ given $\mathcal{F}_{T_1}$, so that
$$Q(A\mid \mathcal{F}_{T_1})Q(C_1\mid \mathcal{F}_{T_1})=Q(A\cap C_1\mid \mathcal{F}_{T_1}),$$
and therefore  by (\ref{eq:equiva6}) the equality (\ref{eq:step1}) follows.\bigskip\\
The general step of the induction works similarly: fixed $n>1$, one assumes
\begin{equation}\label{eq:equiva7n}
Q|_{\mathcal{F}_\infty\vee\sigma (C_1,\ldots, C_{n-1})}=P|_{\mathcal{F}_\infty\vee \sigma (C_1,\ldots, C_{n-1})}
\end{equation}
and, for any choice of $A\in \mathcal{F}_\infty$ and $B\in \sigma (C_1,\ldots, C_{n-1})$, one proves that
\begin{equation}\label{eq:stepn}
P(A\cap B\cap C_n)=Q(A\cap B\cap C_n).
\end{equation}
Due to the arbitrariness of $A$ and $B$, the probability measures $P$ and $Q$ coincide on $\mathcal{F}_\infty\vee\sigma (C_1,\ldots, C_n)$, that is (\ref{eq:equiva2bis}).\bigskip\\
In fact,
\begin{equation*}
P(A\cap B\cap C_n)=E^P\big[P(A\cap B\cap C_n\mid \mathcal{F}_{T_n})\big]=E^Q\big[P(A\cap B\cap C_n\mid \mathcal{F}_{T_n})\big],
\end{equation*}
where since $\mathcal{F}_{T_n}\subset \mathcal{F}_{\infty}$ the second equality follows by  (\ref{eq:equiva1}).\bigskip\\
The immersion property $\mathbb{F}\hookrightarrow\mathbb{F}^{\tau_1}$ implies that, under $P$, $\mathcal{F}_{\infty}$ is conditionally independent of $\mathcal{F}^{\tau_1}_{T_n}$ given $\mathcal{F}_{T_n}$ (see Theorem 3 in \cite{Bre-yor78}).~Therefore, since $A\in \mathcal{F}_{\infty}$ and $B\cap C_n\in\mathcal{F}^{\tau_1}_{T_n}$  (see Proposition \ref{sigma} (a))
\begin{equation}\label{eq:equiva5n}
P(A\cap B\cap C_n)=E^Q\big[P(A\mid \mathcal{F}_{T_n})P(B\cap C_n\mid \mathcal{F}_{T_n})\big].
\end{equation}
Moreover, since $A\in \mathcal{F}_\infty$ and $\mathcal{F}_{T_n}\subset \mathcal{F}_\infty$, (\ref{eq:equiva1}) implies
\begin{equation}\label{eq: step11n}
P(A\mid \mathcal{F}_{T_n})=Q(A\mid \mathcal{F}_{T_n}).\end{equation}
By Condition $(\mathcal{C})$ and by the inclusion $\mathcal{F}_{T^-_n}\subset \mathcal{F}^{\tau_1}_{T^-_n}$,
\begin{equation*}P(B\cap C_n\mid \mathcal{F}_{T_n})=E^P\big[P(B \cap C_n \mid \mathcal{F}^{\tau_1}_{T^-_n})\mid\mathcal{F}_{T^-_n}\big].\end{equation*}
Now, $B\in \mathcal{F}^{\tau_1}_{T^-_n}$  (see Proposition \ref{sigma} (b)) and therefore last expectation coincides with
\begin{equation*}E^P\big[\mathbb{I}_B P(C_n\mid \mathcal{F}^{\tau_1}_{T^-_n})\mid\mathcal{F}_{T^-_n}\big],\end{equation*}
 and since (\ref{eq:equa}) it coincides also with
 \begin{equation*}E^P\big[\mathbb{I}_B Q(C_n\mid \mathcal{F}^{\tau_1}_{T^-_n})\mid\mathcal{F}_{T^-_n}\big].\end{equation*}
 Since $B\in \mathcal{F}^{\tau_1}_{T^-_n}$, the random variable $\mathbb{I}_B Q(C_n\mid \mathcal{F}^{\tau_1}_{T^-_n})$ is measurable with respect to  $\mathcal{F}^{\tau_1}_{T^-_n}\subset \mathcal{F}_\infty\vee \sigma (C_1,\ldots, C_{n-1})$ and hence by the inductive hypothesis (\ref{eq:equiva7n}) next equalities follow
  \begin{align*}P(B\cap C_n\mid \mathcal{F}_{T_n})=&E^Q\big[\mathbb{I}_B Q(C_n\mid \mathcal{F}^{\tau_1}_{T^-_{n}})\mid\mathcal{F}_{T^-_n}\big]=\cr & E^Q\big[Q(B\cap C_n\mid \mathcal{F}^{\tau_1}_{T^-_{n}})\mid\mathcal{F}_{T^-_n}\big]=Q(B\cap C_n\mid \mathcal{F}_{T^-_n}).\end{align*}
Thus, using again the condition $(\mathcal{C})$, we get
\begin{equation}\label{eq: step12n}
P(B\cap C_n\mid \mathcal{F}_{T_n})=Q(B\cap C_n\mid \mathcal{F}_{T_n}).
\end{equation}
The equalities (\ref{eq: step12n}) and (\ref{eq: step11n}) together with (\ref{eq:equiva5n}) imply
\begin{equation}\label{eq:equiva6n}
P(A\cap B\cap C_n)=E^Q\big[Q(A\mid \mathcal{F}_{T_n})Q(B\cap C_n\mid \mathcal{F}_{T_n})\big].
\end{equation}
Finally, (\ref{eq:equiva2}) implies that, also under $Q$, $\mathcal{F}_\infty$ is conditionally independent of $\mathcal{F}^{\tau_1}_{T_n}$ given $\mathcal{F}_{T_n}$, so that
$$Q(A\mid \mathcal{F}_{T_n})Q(B\cap C_n\mid \mathcal{F}_{T_n})=Q(A\cap B\cap C_n\mid \mathcal{F}_{T_n}),$$
and therefore by (\ref{eq:equiva6n}) the equality (\ref{eq:stepn}) follows.
\end{proof}
\subsection{The main result: an $\mathbb{F}^{\tau}$-basis}
\noindent We state now our martingale representation result on $\mathbb{F}^{\tau}$.
\begin{thm}\label{thm:new}
Assume that $\tau$ satisfies Hypothesis $(\mathcal{P})$ and  $(\mathcal{C})$ w.r.t.~$\mathbb{F}$.~If $\mathbb{F}\hookrightarrow\mathbb{F}^{\tau}$, then $\left\{(M^j)_{j\in J}, H^{\tau,\mathbb{F}^\tau}\right\}$ is an $\mathbb{F}^\tau$-basis.\\In particular, $\left\{(M^j)_{j\in J}, H^{\tau,\mathbb{F}^\tau}\right\}$ enjoys the $\mathbb{F}^\tau$-PRP, that is  for any $\mathbb{F}^\tau$-local martingale $Z$  there exists a set of $\mathbb{F}$-predictable processes $\left\{(\Phi^Z_j)_{j\in J},\Psi^Z \right \}$ such that
\begin{equation}\label{eq:repr}Z(t) = Z_0 + \sum_{j\in J}\int_0^t \Phi^Z_j(s)\,dM^j(s) + \int_0^{t} \Psi^Z(s)\,dH^{\tau,\mathbb{F}^\tau}(s),\;t\geq 0,\end{equation}
where $H^{\tau,\mathbb{F}^\tau}$ satisfies (\ref{eq:generalized}).
\end{thm}
\begin{proof}
$\mathbb{F}\hookrightarrow\mathbb{F}^\tau$ implies $\mathbb{F}\hookrightarrow\mathbb{F}^{\tau_1}$ so that point (ii) of Theorem \ref{cor} holds.~Thus, since $\mathbb{F}\hookrightarrow\mathbb{F}^\tau$ implies $H^{\tau_1,\mathbb{F}^{\tau_1}}=H^{\tau_1,\mathbb{F}^\tau}$, $\left\{(M^j)_{j\in J}, H^{\tau_1,\mathbb{F}^\tau}\right\}$ is an $\mathbb{F}^{\tau_1}$-basis.~The random time $\tau_2$ satisfies Hypothesis ($\mathcal{A}$) w.r.t.~$\mathbb{F}^{\tau_1}$.~In fact $\tau_2$ by definition avoids all $\mathbb{F}$-stopping times, and hence avoids also $\tau_1$.~Moreover, $\mathbb{F}^{\tau_1}\hookrightarrow\mathbb{F}^\tau$ (see Proposition 5.4 in \cite{ak-chou-jean-18}).~Thus, we can apply  Theorem 4.9 of \cite{ditella-eng-22} to the progressive enlargement of $\mathbb{F}^{\tau_1}$ with $\tau_2$.~It follows that
$$\left\{(M^j)_{j\in J}, H^{\tau_1, \mathbb{F}^{\tau}}, H^{\tau_2,\mathbb{F}^\tau}\right\}$$
 is an $\mathbb{F}^\tau$-basis, since $\big(\mathbb{F}^{\tau_1}\big)^{\tau_2}=\mathbb{F}^{\tau}$ (see Proposition \ref{rem:composition}).~More precisely, for any $V\in \mathcal{M}^2(P,\mathbb{F}^\tau)$
\begin{align}\label{eq:basis2}
V(t)=V_0+\sum_{j\in J}\int_0^t \phi^V_j(s) d M^j(s)+\int_0^t \gamma^V(s) d H^{\tau_1,\mathbb{F}^\tau}(s)+\int_0^t \xi^V(s) d H^{\tau_2,\mathbb{F}^\tau}(s),\;\;t\geq 0.
\end{align}
Here, $V_0$ is a constant,
$\phi^V_j$, for all $j\in J$, $\gamma^V$ and $\xi^V$ are $\mathbb{F}^\tau$-predictable processes such that the random variables $$\int_0^\infty\,(\phi^V_j(s))^2 d[M^j](s),\;\;\int_0^\infty\,(\gamma^V(s))^2 d[ H^{\tau_1,\mathbb{F}^\tau}](s),\;\;\int_0^\infty\,(\xi^V(s))^2 d[H^{\tau_2,\mathbb{F}^\tau}](s)$$ have finite expectations (see Definition \ref{def:basis}).\bigskip\\
Let $\mathcal{D}$ be the  $\mathbb{F}^\tau$-predictable subset of $\Omega \times [0,\infty)$ defined as
$$\mathcal{D}:=\bigcup_n\,[[T_n]].$$
Now,
 \begin{equation*}\int_{\overline{\mathcal{D}}}dH^{\tau_1,\mathbb{F}^\tau}(t) =0\;\;\;\;\;\; \int_{\mathcal{D}}d H^{\tau_2,\mathbb{F}^\tau}(t) =0, \;\textrm{a.s.}\end{equation*}
The first equality above follows by (\ref{eq-comp222}) and the first equality in (\ref{eq:doppia}).~The second one, instead, follows recalling
that $\tau_2$ avoids $\mathbb{F}$-stopping times and $(T_n)_{n\ge 1}$ is a sequence of $\mathbb{F}$-stopping times.\bigskip\\
 Let $\rho^V$ be the $\mathbb{F}^\tau$-predictable process defined by
 $$\rho^V(t):=\gamma^V(t) \mathbb{I}_{\mathcal{D}}(t)+\psi^V(t) \mathbb{I}_{\overline{\mathcal{D}}}(t).$$
 It holds $E\left[\int_0^\infty\,(\rho^V(s))^2 d[ H^{\tau,\mathbb{F}^\tau}](s)\right]<\infty$.~Then, for all $t\geq 0$
  \begin{equation*}
  \int_0^t \gamma^V(s) d H^{\tau_1,\mathbb{F}^\tau}(s)=\int_0^t \rho^V(s) d H^{\tau_1,\mathbb{F}^\tau}(s)
  \end{equation*}
 and
   \begin{equation*}
     \int_0^t \xi^V(s) d H^{\tau_2,\mathbb{F}^\tau}(s)=\int_0^t \rho^V(s) d H^{\tau_2,\mathbb{F}^\tau}(s).
     \end{equation*}
  Therefore, since $H^{\tau,\mathbb{F}^\tau}= H^{\tau_1,\mathbb{F}^\tau} + H^{\tau_2,\mathbb{F}^\tau}$, equality (\ref{eq:basis2}) can be rewritten as
  \begin{align*}
V(t)=V_0+\sum_{j\in J}\int_0^t \phi^V_j(s) d M^j(s)+\int_0^t \rho^V(s) d H^{\tau,\mathbb{F}^\tau}(s),\;\;t\geq 0.
\end{align*}
 By the arbitrariness of $V$ we conclude that  $\left\{(M^j)_{j\in J}, H^{\tau,\mathbb{F}^\tau}\right\}$ is an $\mathbb{F}^\tau$-basis.~Representation (\ref{eq:repr}) immediately follows by
 Remark \ref{rem:prp basis}.
\end{proof}
\begin{cor}\label{cor:Levy}
Let $\mathbb{F}$ be quasi-left continuous and let $\tau$ satisfy Hypothesis $(\mathcal{P})$.~If $\mathbb{F}\hookrightarrow\mathbb{F}^\tau$, then $\left\{(M^j)_{j\in J}, H^{\tau,\mathbb{F}^\tau}\right\}$ is an $\mathbb{F}^\tau$-basis.
\end{cor}
\begin{proof}
 By Proposition \ref{prop:qlc} Hypothesis $(\mathcal{C})$ for $\tau$ holds true.
\end{proof}
The previous corollary fits well with models dealing with progressive enlargement of the natural filtration of a L\`evy process.~It is worth noting that the natural filtration of any L\`evy process is quasi-left continuous
 and always admits a (possibly non-finite) basis (see Theorem 4.4 in  \cite{ditella-eng-22}).
\begin{remark}\label{cor:Brow}
By Proposition \ref{cor2}  the next result immediately follows.~It is, in fact, a significant extension,  in the case of a single random time, of the seminal result about martingale representations on the progressive enlargement of a \textit{weak Brownian filtration} (see Theorem 2.3 in \cite{kusuoka99}).\vspace{.3em}\\
\textit{Assume to have a  filtration under the usual conditions such that $\mathcal{F}_0$ is trivial and such that it admits a finite basis of continuous martingales.~Consider its progressive enlargement with any random time having nontrivial thin part.~If the immersion property holds, then every local martingale w.r.t.~the enlarged filtration can be represented via stochastic integration of predictable processes w.r.t.~the original basis and the occurrence process of the random time compensated  in the enlarged filtration.~}\vspace{.3em}\\
In particular, this is the case  when the starting filtration is a weak Brownian filtration.
\end{remark}
\begin{example}
Starting from Example \ref{ex:0}, we present an example where all the assumptions of Theorem \ref{thm:new} hold true, but the filtration is not quasi-left continuous.~Let $\mathbb{F}$ be the filtration $\mathbb{A}$ of the Example \ref{ex:0}.~Set
 $ M:=H^{T, \sigma(T\wedge\cdot)}$.~As well-known,
$M$ is a basis for $\mathbb{F}$.~Let $\tau$ be a random time independent of $T$ that takes values in  $\{4,5,6\}$.~As discussed
in the Example \ref{ex:0}, $\mathbb{F}$ is not quasi-left continuous and $\tau$ satisfies both $(\mathcal{P})$ and $(\mathcal{C})$ w.r.t.~$\mathbb{F}$.~Moreover, the independence of $T$ and $\eta$ implies that $\mathbb{F}\hookrightarrow\mathbb{F}^{\tau}$ holds.~
\end{example}
\section{Applications to credit risk}\label{sec5}
\noindent Let us consider the \textit{hybrid credit risk models} studied in \cite{jiao-li18}.~The \textit{hybrid sovereign default model} falls among the possible applications of our result.~In that model, the process driving the solvency process is a Brownian motion and the default is defined as
$\tau:=\zeta^*\wedge\xi$.~The random time $\zeta^*$ describes successive solving downgrades due to economic and political inferences, while $\xi$ stays for the idiosyncratic risk.~From a mathematical point of view, $\zeta^*$ is an accessible random time with a finite enveloping sequence and  $\xi$  is a totally inaccessible random time.\\Remark \ref{cor:Brow} applies with $$\tau_1=\zeta^*\mathbb{I}_{\{\zeta^*<\xi\}}+\infty\mathbb{I}_{\{\zeta^*\ge\xi\}},\;\;\;\;\;\tau_2=\xi\mathbb{I}_{\{\xi<\zeta^*\}}+\infty\mathbb{I}_{\{\xi\ge\zeta^*\}},$$ since the immersion property holds.\\
More generally, Corollary \ref{cor:Levy} applies to the \textit{generalized density models} when the following conditions hold: the reference filtration $\mathbb{F}$ of the asset price process is the natural filtration of any L\'evy process; the critical dates $(T_n)_{n=i,\ldots, N}$ are $\mathbb{F}$-predictable stopping times; as in the hybrid sovereign default model, $\mathbb{F}\hookrightarrow\mathbb{F}^{\tau}$.~It follows that, if the basic information is extended with the observation of the occurrence of the default, then the market can be completed by adding the new tradable asset $H^{\tau,\mathbb{F}^\tau}$ defined as in (\ref{eq:generalized}).
\begin{remark}
Let us recall that in these models the $\mathbb{F}$-reduction  of the compensator of $\tau$, $\Lambda^{\tau,\mathbb{F}}$, satisfies
  \begin{equation}\label{eq:generalized2}
  \Lambda^{\tau,\mathbb{F}}=\sum_{n\ge 1}\frac{P(C_n\mid\mathcal{F}_{T_n^-})}{G_{T_n^-}}\mathbb{I}_{\{T_n\le \cdot\}}+\int^{ \cdot}_0\frac{\alpha_s(s)}{G_{s^-}}\;ds,
  \end{equation}
   where $\alpha(\cdot)$ is the  $\mathbb{F}$-generalized density of $\tau$ and $G$ is the $\mathbb{F}$-Az\'ema super-martingale of $\tau$, that is $G_t=P(\tau >t\mid \mathcal{F}_t), t\ge 0$.~Comparing  (\ref{eq:generalized2}) with (\ref{eq:comp}) we derive that in the generalized density models under our conditions the  $\mathbb{F}$-thick part of $\tau$ has \textit{$\mathbb{F}$-intensity rate} (see Definition 2.20, p.~39 in \cite{aksamit-jean-17}) which coincides with $\frac{\alpha_\cdot}{G_{\cdot^-}}$, and moreover, for all  $n\geq 1$,
\begin{equation*}
 \frac{P(C_n\mid\mathcal{F}_{T_n^-})}{G_{T_n^-}}\mathbb{I}_{\{T_n\le \tau\wedge\cdot\}}= P(C_n\mid\mathcal{F}^{\tau}_{T_n^-})\mathbb{I}_{\{T_n\le \tau\wedge\cdot\}}.
   \end{equation*}
  \end{remark}
\vspace{1.0 em}
\noindent The generalization of the classical \textit{Cox procedure} for default times provides a second class of applications of our results (see \cite{gueye-jeanb}).~Let $(T_n)_{n\ge 1}$ be an increasing sequence of $\mathbb{F}$-predictable stopping times, $(\theta_n)_{n\ge 1}$ a sequence of random variables such that, for any $n\ge 1$, $\theta_n$ is $\mathcal{F}_{T_n^{-}}$-measurable, $k$ an $\mathbb{F}$-adapted non-negative process,  $\Theta$ an exponential random variable of parameter one, independent of  $\mathcal{F}_\infty$.~Set
$$\tau:=\inf\{t\geq 0, K_t\geq \Theta\},$$
where
\begin{equation*}K_t:=\sum_{n\ge 1}\theta_n\mathbb{I}_{\{T_n\le  t\}}+\int_0^ t k_s\,ds,\;\;t\ge 0.\end{equation*}
If $\mathbb{F}$ is a natural L\'evy filtration then all assumptions of Corollary \ref{cor:Levy} are satisfied (see p.~471 and Section 3.2.5 in \cite{gueye-jeanb}).~If $\mathbb{F}$ represents the basic information of a complete market, then one gets a new complete market with extended information given by  $\mathbb{F}^\tau$ by adding the new tradable asset $H^{\tau,\mathbb{F}^\tau}$ defined as in (\ref{eq:generalized}).\bigskip\\
\noindent The last example of application deals with a model that has been proposed in \cite{MR4152276} studying the market's completeness under filtration shrinkage and in \cite{ak-chou-jean-18} as an example of thin time.\\
Let $B$ be the \textit{L\'evy transformation of a Brownian motion} $W$, that is the process $$B:=\int_0^\cdot sign(W_s)\;dW_s,$$ and let $\mathbb{F}^B$ and $\mathbb{F}^W$ be the natural filtrations of $B$ and $W$, respectively.~As well-known $B$ is a natural Brownian motion, and $\mathbb{F}^B=\mathbb{F}^{|W|}\subsetneq\mathbb{F}^W$, and the exponential martingale $\mathcal{E}(B)$ is both an $\mathbb{F}^B$-basis and an $\mathbb{F}^W$-basis.\\
As observed in \cite{MR4152276}, in the language of finance modeling, $(\mathcal{E}(B),\mathbb{F}^B)$ is a complete market, as well as $(\mathcal{E}(B),\mathbb{F}^W)$, but $\left(\mathcal{E}(B),(\mathbb{F}^B)^\tau\right)$ is not when
\begin{equation}\label{def:tau ex}
\tau:=\inf\{t\geq 0: W_t=1\}.
\end{equation}
In fact, if $(T_n)_{n\geq 1}$ is the sequence of $\mathbb{F}^B$-stopping times defined by
$$T_n:=\inf\{t> S_{n-1}: |W_t|=1\}$$
with
$$S_0:=0,\;\;\;\text{and}\;\;\;S_n:=\inf\{t> T_{n-1}: |W_t|=0\},\;n>1,$$
then the process
\begin{equation*}
N:=\mathbb{
I}_{\{\tau\le
  \cdot\}}-\frac12\sum_{n\geq 1}\mathbb{
I}_{\{T_n\le \tau\wedge
  \cdot\}},
\end{equation*}
is a discontinuous $(\mathbb{F}^B)^\tau$-martingale and cannot be represented by integration w.r.t.~$\mathcal{E}(B)$.\bigskip\\
Thus, since $\mathbb{F}^B\subset \big(\mathbb{F}^B\big)^\tau\subset\mathbb{F}^W$, the authors of \cite{MR4152276} use the example to explain that some markets could lose completeness after the reduction of information or, vice versa, when the addition of information is too little.~Instead, in \cite{ak-chou-jean-18} by the previous example the authors discuss a case of a naturally totally inaccessible random time that coincides with its thin part w.r.t.~a given filtration, and has an exhausting sequence of predictable stopping times.~We stress that $\tau$, although naturally totally inaccessible, is $(\mathbb{F}^B)^\tau$-accessible and becomes predictable in the even larger filtration $\mathbb{F}^W$.\bigskip\\
Here we answer the question:  how can be represented all $(\mathbb{F}^B)^\tau$-local martingales since $\mathcal{E}(B)$ is not enough?\vspace{.1em}
\begin{proposition} $(\mathcal{E}(B),N)$ is an $(\mathbb{F}^B)^\tau$-basis.\end{proposition}
\begin{proof}
We apply Remark \ref{cor:Brow} to the Brownian filtration $\mathbb{F}^B$,  the $\mathbb{F}^B$-basis $\mathcal{E}(B)$ and $\tau$ given by (\ref{def:tau ex}).~The assumption $\mathbb{F}^B\hookrightarrow (\mathbb{F}^B)^\tau$ holds since $\mathbb{F}^B\hookrightarrow\mathbb{F}^W$.~Let us prove that $H^{\tau, (\mathbb{F}^B)^{\tau}}=N$.\bigskip\\
By Remark (\ref{rem:remcomp})
\begin{equation*}
H^{\tau, (\mathbb{F}^B)^\tau}=\mathbb{
I}_{\{\tau\le
  \cdot\}}-\sum_{n\geq 1}P\left(C_n\mid(\mathcal{F}^B)^\tau_{T_n^-}\right)\mathbb{I}_{\{T_n\le\tau\wedge
  \cdot\}}.
\end{equation*}
Now, for any fixed $n\geq 1$, recalling Proposition \ref{sigma}
\begin{equation*}
P\left(C_n\mid(\mathcal{F}^B)^\tau_{T_n^-}\right)=P\left(W_{T_n}=1\mid\mathcal{F}^{|W|}_{T_n}\vee\sigma (C_1,\ldots, C_{n-1})\right),\end{equation*}
and $$\mathcal{F}^{|W|}_{T_n^-}=\mathcal{F}^{B}_{T_n^-}=\mathcal{F}^{B}_{T_n}=\mathcal{F}^{|W|}_{T_n}.$$
Moreover,
$$P\left(W_{T_n}=1\mid\mathcal{F}^{|W|}_{T_n}\vee\sigma (C_1,\ldots, C_{n-1})\right)=\frac12,$$
since, by the symmetry of the Brownian motion,
$$P\left(W_{T_n}=1\mid\mathcal{F}^{|W|}_{T_n}\vee\sigma (C_1,\ldots, C_{n-1})\big)=P\big((W_{T_n}=1)^c\mid\mathcal{F}^{|W|}_{T_n}\vee\sigma (C_1,\ldots, C_{n-1})\right).$$
In fact,
$$P\left((W_{T_n}=1)^c\mid\mathcal{F}^{|W|}_{T_n}\vee\sigma (C_1,\ldots, C_{n-1})\right)=P\left(-W_{T_n}=1\mid\mathcal{F}^{|-W|}_{T_n}\vee\sigma (C_1,\ldots, C_{n-1})\right),$$
and
$$\sigma (C_1,\ldots, C_{n-1})=\sigma (D_1,\ldots, D_{n-1}),$$
where $$D_h:=\{-W_{T_h}=1\}=\{W_{T_h}=1\}^c,\;h=1,\ldots,n-1.$$
\end{proof}
\section{Conclusions and perspectives}\label{sec6}
The thin-thick approach leads to establish a set of easy-to-read assumptions that guarantee the propagation of the PRP on the filtration progressively enlarged with a random time $\tau$.~It also highlights a representation for the compensator of $\tau$ not yet studied in the literature.~In financial terms, our result proves that, for a large class of credit risk models, the market with the information extended by the observation of the default can be completed by introducing a new tradable asset.~\bigskip\\
Two are the  natural developments of our research.~On the one hand, to extend the results obtained to the case where the family $(M^j)_{j\in J}$ contains a non-finite number of elements, possibly  not square-integrable.~On the other hand,
 to derive martingale representations on the enlarged filtration relaxing either the assumption on the thin part of $\tau$ or the continuity condition for the filtration.~In particular, dropping the continuity condition new elements related to the covariation processes should appear in the basis.~This should cause the growth of the Jacod dimension of the enlarged filtration, that is of the minimal possible dimension of a vector-valued local martingale enjoying the PRP w.r.t.~the enlarged filtration.~In fact, we think that the covariation process of any pair of orthogonal  local martingales, when nontrivial,  can not be represented as a vector stochastic integral w.r.t.~the pair itself.~The validation of the above  conjectures is the topic of ongoing research.
\section*{Acknowledgments}
\noindent The authors thank an anonymous referee for having highlighted the true meaning of their assumptions.~This helped improve the presentation of the work.~The authors acknowledge  the MIUR Excellence Department Project MatMod@TOV awarded to the Department of Mathematics, University of
Rome Tor Vergata CUP E83C23000330006.
\bibliography{{biblio}}
\end{document}